\documentclass{amsart}
\usepackage{amsmath,amsfonts,amssymb,amscd,amsthm,epsf}
\usepackage{pinlabel}
\newtheorem{prop}{Proposition}[section]
\newtheorem{thm}[prop]{Theorem}
\newtheorem{cor}[prop]{Corollary}
\newtheorem{ques}[prop]{Question}
\newtheorem{quess}[prop]{Questions}

\theoremstyle{definition}
\newtheorem{de}[prop]{Definition}

\newtheorem{examples}[prop]{Examples}
\theoremstyle{remark}
\newtheorem{Remarks}[prop]{Remarks}             
\def\C{{\mathbb C}}

\def\Z{{\mathbb Z}}
\def\R{{\mathbb R}}

\def\int{\mathop{\rm int}\nolimits}

\def\id{\mathop{\rm id}\nolimits}

\def\SO{\mathop{\rm SO}\nolimits}

\def\co{\colon\thinspace}

\begin{document}
\title{Quotient manifolds of flows}
\author{Robert E. Gompf}
\address{The University of Texas at Austin, Mathematics Department RLM 8.100, Attn: Robert Gompf,
2515 Speedway Stop C1200, Austin, Texas 78712-1202}
\email{gompf@math.utexas.edu}
\begin{abstract} This paper investigates which smooth manifolds arise as quotients (orbit spaces) of flows of vector fields. Such quotient maps were already known to be surjective on fundamental groups, but this paper shows that every epimorphism of countably presented groups is induced by the quotient map of some flow, and that higher homology can also be controlled. Manifolds of fixed dimension arising as quotients of flows on Euclidean space realize all even (and some odd) intersection pairings, and all homotopy spheres of dimension at least two arise in this manner. Most Euclidean spaces of dimensions five and higher have families of topologically equivalent but smoothly inequivalent flows with quotient homeomorphic to a manifold with flexibly chosen homology. For $m\ge 2r>2$, there is a topological flow on $(\R^{2r+1}-8{\rm\ points})\times\R^m$ that is unsmoothable, although smoothable near each orbit, with quotient an unsmoothable topological manifold.
\end{abstract}
\maketitle


\section{Introduction}

To understand the global properties of flows induced by vector fields on manifolds, it seems natural to study the {\em quotient space} ({\em orbit space}) of a flow, the topological space whose points are the trajectories of the flow. This invariant has perhaps been neglected, since such spaces frequently fail to be Hausdorff or even $T_1$. However, one can still proceed in several directions. Even without the $T_1$ condition, one can still apply algebraic topology: Calcut, McCarthy and the author \cite{CGM} showed that the quotient map induces a surjection of fundamental groups (provided that the quotient is semilocally simply connected). For gradient-like vector fields, Calcut and the author \cite{CG} showed that the quotient map is a weak homotopy equivalence onto a locally contractible space. In the present paper we take a complementary approach, focusing on flows whose quotients have the nicest possible local properties, so are metrizable manifolds with induced smooth structures. That is, we require the simplest possible dynamics allowing closed orbits (and sometimes fixed points). In the leading question of his problem list \cite{A}, Arnol'd observed that every smooth manifold $R$ homeomorphic to $\R^4$ is the quotient of a flow on $\R^5$, since $R\times\R$ must be diffeomorphic to $\R^5$. Since there are uncountably many diffeomorphism types of such manifolds $R$ (the {\em exotic smoothings} of $\R^4$, see \cite{infR4}, \cite{GS}), we obtain uncountably many flows on $\R^5$ that are topologically trivial (conjugate by a homeomorphism to the product flow on $\R^4\times\R$) but smoothly distinguished from each other by the diffeomorphism type of the quotient space. This motivates the more general question of which $n$-manifolds can arise as quotients of flows on $\R^{n+1}$ or other simple open manifolds. One might expect that such quotients must be open, and perhaps, in light of the above homotopy theoretic results, that a manifold quotient of $\R^{n+1}$ must be contractible. However, Calcut \cite{CGM} constructed a quadratic vector field on $\R^{2m+1}$ with quotient $\C P^m$. The present paper shows that this is a much more general phenomenon. We construct a wide variety of smooth $n$-manifolds, both open and closed, as quotients of $\R^{n+1}$ and other simple manifolds, with quotients exhibiting a broad range of homology, intersection forms and exotic smooth structures. 

The manifold condition on quotients tightly constrains the local structure of flows, but all of our examples are locally even more tightly controlled. To create interesting topology, we allow circles as orbits. However, we require their Poincar\'e maps to be the identity, so they have neighborhoods realized as trivial circle bundles. All of the compact orbits appearing in Section~\ref{2h} have this form. While the homotopy types of circle bundles are highly constrained by the long exact homotopy sequence, our flows with only some circle orbits have much more flexibility, allowing us to control the fundamental group and second Betti number of the quotient, subject only to $\pi_1$-surjectivity of the quotient map. (The same method also yields examples of higher-dimensional foliations; see the final paragraph of this paper.) In Section~\ref{Higher}, we gain control of higher homology by also allowing fixed points. The manifold quotient condition tightly constrains the fixed set. However, we canonically obtain a smooth manifold quotient from a codimension-4 fixed set whose meridians are 3-spheres with the Hopf circle action: The meridians descend to meridian 2-spheres, so the fixed set descends to have codimension three. The quotient map has a quadratic local model, and can be thought of as a generalization of a 2-fold branched covering map.

To construct our examples, we consider quotient manifolds as handlebodies (with open handles). Calcut exhibited $\C P^m$ as the quotient of a flow on $\R^{2m+1}$ by starting with the Hopf fibration $S^{2m+1}\to\C P^m$ and removing a point from the domain. Note that $\C P^m$ is built with a handle of each even index through $2m$, but the lifts of all but the 0-handle are ``invisible'' in the domain $\R^{2m+1}$. Since quotient maps of flows are always surjective on $\pi_1$ (when the quotient is a manifold and hence semilocally simply connected), it is impossible to have 1-handles that are invisible in this manner. However, close inspection of the invisible 2-handle in Calcut's example reveals how it functions, allowing us to construct such 2-handles more generally. We investigate such invisible 2-handles in Section~\ref{2h}. Then we construct higher-index handles in Section~\ref{Higher} by replacing the Hopf map $S^3\to\C P^1=S^2$ by its iterated suspensions. The resulting invisible $k$-handles have nonempty fixed-point sets of the sort described previously, but allow us to control higher homology of the smooth quotient. It remains an open question whether the higher-index handles of the Hopf map to $\C P^m$ can be exploited to yield fixed-point free flows with invisible $k$-handles for even $k>2$ (see Section~\ref{Further}).

Our main theorems assert that if an $n$-manifold $V$ is obtained from another manifold $U$ by attaching a layer of open handles of index at least two, then it is the quotient of a flow on $U\times\R$. (See Theorem~\ref{main} for 2-handles and fixed-point free flows, and Theorem~\ref{high} in general.) We immediately see (Corollaries~\ref{contr} and \ref{highapps}(a)) that when $U$ is contractible, $V$ is the quotient of some flow on $\R^{n+1}$.  This gives almost complete control of the homology and intersection pairing of quotient manifolds of flows on $\R^{n+1}$ (Example~\ref{examples}(a) and Corollary~\ref{highapps}(b)) although they must be simply connected. We also obtain  every compact homotopy $n$-sphere ($n\ge2$) as such a quotient (Examples~\ref{spheres}(a,c)). Section~\ref{2h} gives further applications of invisible 2-handles, which are expanded to higher homology in Section~\ref{Higher} (Corollary~\ref{highapps} and Remarks~\ref{highapprem}). Corollary~\ref{group} gives a counterpoint to $\pi_1$-surjectivity of quotient maps: {\em Every} epimorphism of countably generated groups is induced by the quotient map of some flow with manifold quotient (of fixed dimension four or larger), and for many manifolds $U$, there are such flows on $U\times\R$ realizing all quotients of $\pi_1(U)$. Corollary~\ref{all} observes that every oriented $n$-manifold admitting a proper Morse function, bounded below, whose indices are at most two, is a quotient of a flow on $(\R^2-S)\times\R^{n-1}$ for some discrete, closed $S$, with a similar result in the nonorientable case. In particular, every oriented surface (possibly closed or with infinite topology) is a quotient of a flow on $\R^3$ minus parallel lines, every oriented, open 3-manifold is a quotient of $\R^4$ minus parallel planes, and various related results hold for 4- and 5-manifolds (Examples~\ref{low}).

The tools developed in this paper allow a deeper study of exotic flows of the sort first introduced in dimension five by Arnol'd. We call two flows {\em topologically} (respectively {\em smoothly}) {\em equivalent} if they are conjugate by a homeomorphism (resp.\ diffeomorphism). ``Most'' smooth 4-manifolds $X$ have infinitely many diffeomorphism types of exotic smoothings (uncountably many in the noncompact case), such that the corresponding smoothings on $X\times\R$ are diffeomorphic. As Arnol'd observed, we can interpret these as topologically equivalent but smoothly inequivalent flows on the 5-manifolds $X\times\R$, distinguished by the diffeomorphism types of their quotients. Since these are product flows, their quotient maps are homotopy equivalences. In contrast, we show (Example~\ref{examples}(b), cf.\ also Corollary~\ref{last}) that $\R^5$ has infinitely many topologically equivalent but smoothly inequivalent flows with quotient realizing any preassigned intersection form, and some forms are realized by uncountable families. All of these 5-dimensional families, like those of Arnol'd, are unstable under Cartesian product with $\R$. That is, the flows within each family become smoothly equivalent if we take their product with the flow of velocity zero on $\R$ (or any other manifold of positive dimension). Using handles of higher index (Examples~\ref{spheres}) we can instead construct finite families of exotic flows on $\R^{n+1}$ that are stable under product with $\R$ (although their behavior under iterated stabilization, i.e., product with $\R^m$, is less clear). On the other hand, Example~\ref{spheres}(e) exhibits unsmoothable topological flows on smooth manifolds, i.e., 1-parameter groups of homeomorphisms that are not conjugate by a homeomorphism to a smooth flow on any manifold. We find such flows on $(\R^{2r+1}-8{\rm\ points})\times\R^m$ whenever $m\ge 2r>2$ that are unsmoothable since their quotient manifolds are, but are smoothable in a neighborhood of each orbit.

Except where otherwise stated, all manifolds in this paper are assumed to be metrizable, smooth, connected and without boundary. Disks, particularly the disks $D^k$ and $2D^k$ of radius one and two, are compact, and handlebodies have boundaries coming from the boundaries of their compact handles. However, we will also consider open handles attached at infinity to open manifolds. Embeddings are smooth, injective immersions, and need not be proper. An annulus is the difference of two concentric open or closed disks of the same dimension, diffeomorphic to the product of a sphere with an interval or ray. Flows will be smooth ($C^\infty$) and complete (defined for all time) with fixed-point sets of codimension four, and those of Section~\ref{2h} will be fixed-point free. The quotient manifolds will necessarily be orientable if and only if the domains are orientable.


\section{Invisible 2-handles}\label{2h}

Our main results are obtained by modifying a flow to add handles to its quotient without changing the diffeomorphism type of the manifold supporting the flow. We begin by adapting the notion of attaching handles to the setting of open manifolds, in a way that is both more general than the compact version and more convenient for our purposes.

\begin{figure}
\labellist
\small\hair 2pt
\pinlabel $D^k$ at 93 192
\pinlabel $2D^k$ at 92 212
\pinlabel $\R^k$ at 22 183
\pinlabel $\R^{n-k}$ at -2 164
\pinlabel $\R^k\times\R^{n-k}$ at 90 120
\pinlabel $U$ at 23 38
\pinlabel 0 at -3 145
\pinlabel {End of $U$} at 84 65
\pinlabel $i$ at 193 42
\pinlabel $j$ at 186 154
\pinlabel $i(U)$ at 185 103
\pinlabel $j(\R^k\times\R^{n-k})$ at 307 164
\pinlabel $V$ at 228 174
\pinlabel {Annulus $j((2D^k-D^k)\times\{0\})$} at 256 22
\endlabellist
\centering
\includegraphics{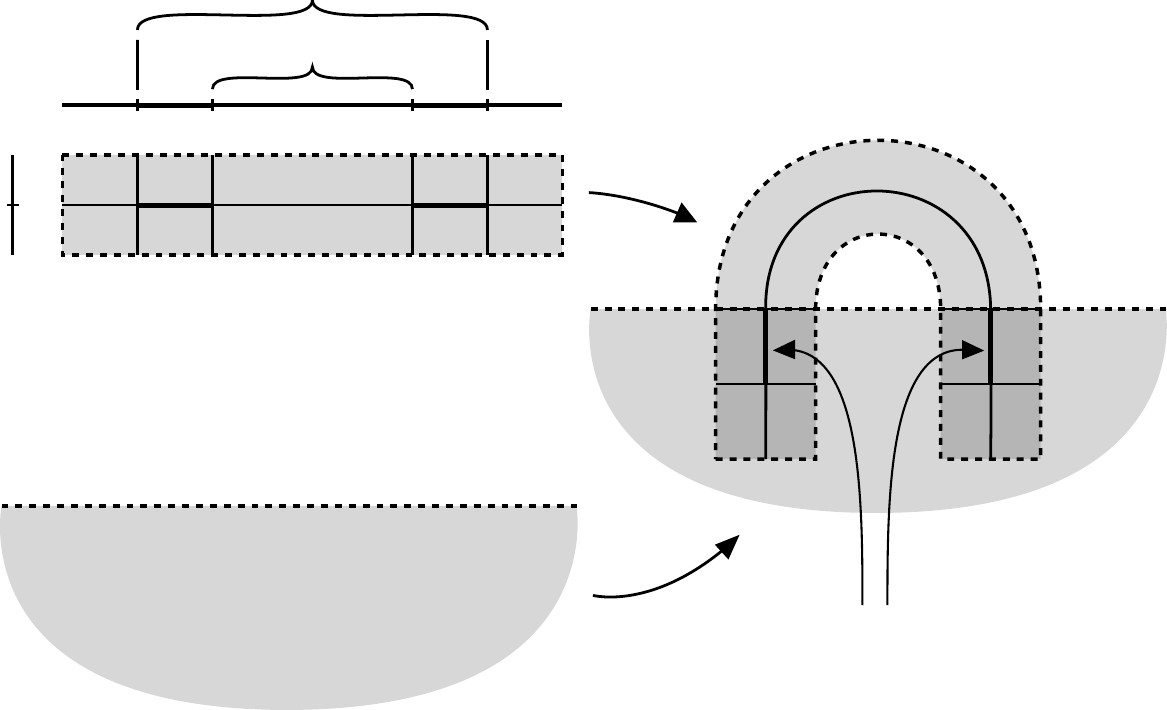}
\caption{Obtaining a manifold $V$ from $U$ by attaching a $k$-handle at infinity.}
\label{handle}
\end{figure}
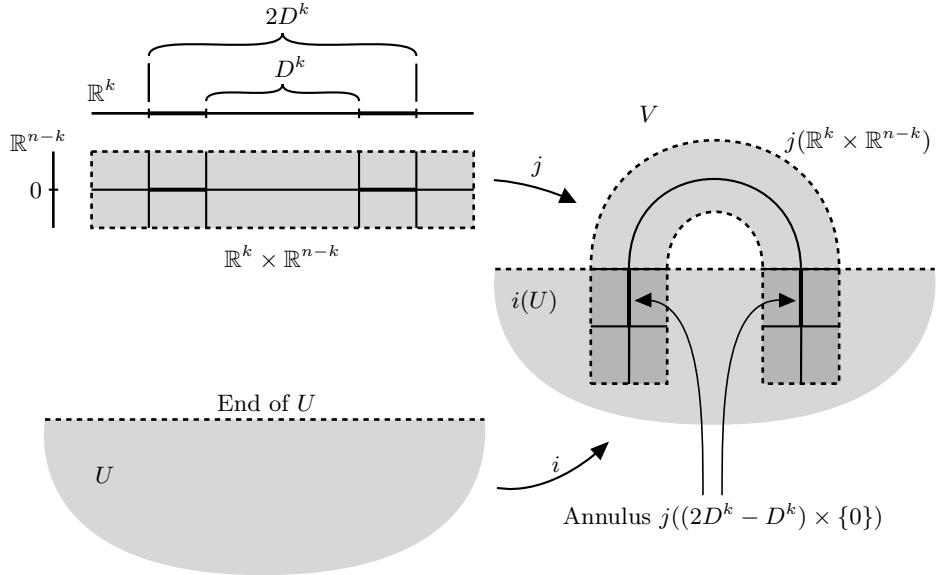

\begin{de} (See Fig.~\ref{handle}.) An $n$-manifold $V$ is obtained from another $n$-manifold $U$ by {\em attaching a layer of handles at infinity} if there are embeddings $i$ of $U$ and $j$ of a (possibly infinite) disjoint union of copies of $\R^k\times\R^{n-k}$ into $V$, with images covering $V$, such that $j^{-1}(i(U))$ consists of all of the copies of $(\R^k-D^k)\times\R^{n-k}$. Each  copy of $\R^k\times\R^{n-k}$ will be called an {\em open $k$-handle}, and if the index $k$ is held fixed, $V$ is obtained by {\em attaching $k$-handles at infinity}.
\end{de}

\noindent Thus, $j$ embeds each copy of $(\R^k-D^k)\times\R^{n-k}$ into $i(U)\approx U$ as a tubular neighborhood  of the annulus $j((\R^k-D^k)\times\{0\})$. The resulting proper embeddings of the half-open annulus $2D^k-D^k$ into $U$, with their induced normal framings, are enough to reconstruct $V$ up to diffeomorphism. In fact, any countable collection of disjoint, proper, framed embeddings of half-open annuli determines a manifold obtained from $U$ by adding a layer of handles. (After suitably truncating the embeddings, one can assume that every compact subset of $U$ intersects only finitely many annuli.) The usual notion of attaching a layer of handles to the boundary of a compact manifold has the effect of attaching handles at infinity to the interior. However, the noncompact notion is more flexible. For example, we can simultaneously attach infinitely many handles or attach to ends with infinite topology. The simplest case of the latter is to join two components of a disconnected $U$ by a 1-handle. The result is called an {\em end-sum} of the two components \cite{infR4}, \cite{GS} or a {\em connected sum at infinity} \cite{CH}. End-summing with an exotic $\R^4$ (with infinite topology) is useful for creating exotic smooth structures on open 4-manifolds. For a detailed study of 1-handles at infinity, see \cite{CG2}.

\begin{examples}\label{handles} (a) For a (connected) $n$-manifold $X$ with a proper Morse function $\varphi\co X\to[0,\infty)$, we obtain a nest $\R^n=X_0\subset X_1\subset\cdots\subset X_n=X$, where each $X_k$ with $k>0$ is obtained from $X_{k-1}$ by adding a layer of $k$-handles at infinity, which correspond bijectively to the index-$k$ critical points of $\varphi$ after extra index-0 critical points have been canceled. This is immediate from standard handle theory if there are only finitely many critical points, and the general case follows from the same ideas: Use induction on decreasing $k$. For $k>0$, let $X_k$ be the complement of the ascending manifolds of the critical points of index exceeding $k$. For $X_0$, cut $X_1$ back to a (thickened) maximal tree and note that a nested union of $n$-balls is diffeomorphic to $\R^n$ with a standard nest of round balls. For later use, note that the attaching rays of the 1-handles can be arranged to be radial in this structure, avoiding the difficulty that when $n=3$, uncountably many nondiffeomorphic contractible 3-manifolds can be made by end-summing two copies of $\R^3$ along knotted rays (Myers \cite{M}).

(b) Any symmetric bilinear form over a countably generated, free $\Z$-module can be realized as the intersection form of an open 4-manifold obtained from $\R^4$ by adding 2-handles at infinity. To see this, first inductively construct a family $\{L_m|m\in\Z^{\ge0}\}$ of framed links in $S^3$, with $L_0$ empty and each subsequent $L_m$ obtained from $L_{m-1}$ by adding a knot $K_m$, so that the linking pairing of $L_m$ is the given bilinear form restricted to the first $m$ generators. Then attach a 2-handle along each annulus $K_m\times [m,\infty)\subset S^3\times (0,\infty)\approx \R^4-\{0\}\subset\R^4$.

(c) There are open 4-manifolds that cannot be interiors of compact 4-manifolds with boundary, but are obtained (for example) by attaching a single 2-handle to $\R^4$ at infinity. Let $X^4$ be the interior of a handlebody whose handles all have index two or less. (For a simple example, consider an $\R^2$-bundle over $S^2$, coming from a 0-handle with a 2-handle attached.) Then we can typically construct exotic smoothings on $X$ \cite{MinGen} by replacing the 2-handles by {\em Casson handles}, which are homeomorphic to open 2-handles. One can distinguish infinitely many diffeomorphism types of such smoothings on $X$ if, for example, $H_2(X)\ne 0$, and sometimes one can distinguish uncountably many such types. The operation of replacing 2-handles by Casson handles can alternatively be described as reattaching ordinary 2-handles at infinity along exotic proper annuli: Each 2-handle $h$ of $X$ can be rewritten as a 1-handle and a pair of 2-handles $h_1$ and $h_2$, each going once over the 1-handle, so that the 1-handle cancels either 2-handle and leaves the other 2-handle as $h$. If we replace $h_1$ by a Casson handle $CH$, then $h_2$ cancels the 1-handle, and the net result is to replace $h$ by $CH$ as desired. On the other hand, we can think of this new smoothing as being obtained from $X-h$ by attaching the 1-handle and Casson handle to obtain an intermediate open manifold $X_0$, then attaching $h_2$ to $X_0$ at infinity. We can absorb the 1-handle into $CH$, so that $X_0$ is obtained by attaching $CH$ to $X-h$ along only half of its attaching region. Thus, $X_0$ is the end-sum of $X-h$ with $\int CH$. But every Casson handle interior is diffeomorphic to $\R^4$ (by Freedman \cite{F}, cf.\ also \cite{GS}~Exercises~9.4.1(b,c)), so $X_0$ is diffeomorphic to $X-h$, and the new smoothing is obtained by attaching $h_2$ to $X-h$ at infinity. As a simple example, the $\R^2$-bundle over $S^2$ with Euler class one admits uncountably many such smoothings (\cite{MinGen}~Theorem~7.7), each obtained by attaching a 2-handle to $\R^4$ at infinity, and which cannot be interiors of compact smooth manifolds with boundary. (Thus, any proper Morse function on one of these requires infinitely many critical points.) The 2-handle attaches along a proper annulus in $\R^4$ that cannot arise from a compact annulus in any compactification of $\R^4$ as a smooth ball, even though the proper annulus becomes the standard $\R^2-\int D^2$ under some (nonsmooth) self-homeomorphism of $\R^4$.
\end{examples}

We now state our main theorem for constructing invisible 2-handles. This is motivated by the invisible 2-handle implicit in Calcut's flow with quotient $\C P^1$ (cf.\ Section~\ref{Higher}).

\begin{thm}\label{main} Suppose that an $n$-manifold $V$ is obtained from another $n$-manifold $U$ by attaching 2-handles at infinity. Then there is a smooth, fixed-point free flow on $U\times\R$ with quotient $V$ and quotient map homotopy equivalent to the embedding $i\co U\hookrightarrow V$.
\end{thm}

\begin{proof} We lift each 2-handle by taking its product with $S^1$ and gluing along a diagonal circle generating the resulting homology. Let $(r,\theta)$ denote the polar coordinate functions on $\R^2$, and choose an embedding $\iota\co\R\to S^1=\R/\Z$ with $\iota(0)=0$. Each 2-handle of $V$ is attached to $U$ using an embedding $\psi\co(\R^2-D^2)\times \R^{n-2}\to U$ with image $W$. For each 2-handle, glue a copy of $\R^2\times\R^{n-2}\times S^1$ onto $U\times\R$ by the embedding $W\times\R\to\R^2\times\R^{n-2}\times S^1$ given by $(w,t)\mapsto(\psi^{-1}(w),\iota(t)+\theta\circ\psi^{-1}(w))$. On the resulting $(n+1)$-manifold $X$, the tangents to the $\R$ and $S^1$ fibers fit together into a global oriented line field, which supports a complete vector field whose flow has quotient $V$. To diffeomorphically identify $X$ with $U\times\R$, compose each gluing map with the self-diffeomorphism of $\R^2\times\R^{n-2}\times S^1$ given by $(r,\theta,x,\phi)\mapsto (r,\theta-\phi,x,\phi)$. Then the annulus $\psi((\R^2-D^2)\times\{0\})\times\{0\}$ in $U\times\R$ is sent to the annulus $r>1$, $\theta=0$, $x=0$, $\phi$~arbitrary in $\R^2\times\R^{n-2}\times S^1$. Since the latter manifold is just a tubular neighborhood of the annulus, it is routine to construct a diffeomorphism absorbing it into $U\times\R$, which is then identified with $X$.
\end{proof}

\begin{cor}\label{contr} If an $n$-manifold $V$ is obtained from a contractible manifold $U$ by attaching 2-handles at infinity, then $V$ is the quotient of some fixed-point free flow on $\R^{n+1}$.
\end{cor}

\noindent For each $n\ge 3$, there are uncountably many homeomorphism types of contractible, smooth $n$-manifolds (see Glaser \cite{G}), distinguished by their fundamental group at infinity.

\begin{proof} It suffices to observe that $U\times\R$ must be diffeomorphic to $\R^{n+1}$. This follows by combining work of McMillan and Perelman ($n=3$) and Stallings ($n>3$). See \cite{CN} for details.
\end{proof}

\begin{examples}\label{examples} (a) By Example~\ref{handles}(b) and Theorem~\ref{main}, there is a flow on $\R^5$ whose quotient is an open 4-manifold $V$ with intersection form realizing any preassigned symmetric bilinear form over a countably generated, free $\Z$-module. If the module is finitely generated, then $V$ is the interior of a compact manifold with boundary, namely, a 4-ball with finitely many 2-handles attached.

(b) Let $X$ be the interior of a 4-manifold obtained by attaching 2-handles to a 4-ball. Then all of the exotic smoothings of $X$ made by replacing 2-handles by Casson handles as in Example~\ref{handles}(c) are quotients of flows on $\R^5$. Choosing $X$ more generally as in (a) above, we can realize any symmetric form over a countably generated, free $\Z$-module by infinitely many topologically equivalent but smoothly distinct flows on $\R^5$. (The rank-0 case is covered by Arnol'd's examples.) If the module is not finitely generated, we obtain uncountably many such flows (see \cite{MinGen}). Uncountable families for some finitely generated modules can be constructed similarly, from \cite{MinGen} Theorem~7.1 ($\widetilde X=X$) or 7.7.

(c) In the above examples, we can replace the 4-ball by any other contractible 4-manifold. (For distinguishing flows as in (b), assume no 3-handles.) Many contractible 4-manifolds require 1-handles in any handle decomposition. Thus, a quotient of $\R^5$ by a flow may be forced to have 1-handles, even though it must be simply connected.
\end{examples}

\begin{cor}\label{group} (a) For every $n\ge 4$, countably generated group $G$ and epimorphism $f\co G\to H$, there is an oriented $(n+1)$-manifold with a fixed-point free flow for which the quotient is an $n$-manifold and the $\pi_1$-epimorphism induced by the quotient map is given by $f$.

(b) Let $U$ be an oriented manifold of dimension $n\ge4$ with a proper Morse function $\varphi\co U\to\R$ whose critical points have index at most $n-2$. Then $U\times\R$ has flows as in (a) realizing every quotient of $\pi_1(U)$.
\end{cor}

\noindent Note that $\varphi$ need not be bounded below, so $U$ is constructed from some (possibly empty) $I\times X^{n-1}$ by attaching (possibly infinitely many) handles of index at most $n-2$ above and of index at least two below.

\begin{proof}  For (a), note that $G$ is a countable set. Since $n\ge 4$, there is an $n$-manifold $U$ with $\pi_1(U)\cong G$. (End-sum together a copy of $S^1\times\R^{n-1}$ for each generator, then add a 2-handle at infinity for each relator.) Let $V$ be obtained from $U$ by adding further 2-handles at infinity to kill $\ker f$. Then $\pi_1(V)\cong H$ and inclusion $U \hookrightarrow V$ induces $f$ on $\pi_1$. Apply Theorem~\ref{main}. The proof of (b) is similar, where the annuli for attaching the 2-handles of $V$ come from generating circles for $\ker f$ by following the gradient flow of $\varphi$ upward, avoiding critical points via the given index constraint.
\end{proof}

\begin{cor}\label{all} If an oriented $n$-manifold $V$ admits a proper Morse function $\varphi\co V\to[0,\infty)$ whose critical points have index at most two, then it is the quotient of a fixed-point free flow on $(\R^2-S)\times\R^{n-1}$, where $S$ is a discrete, closed subset no larger than the set of index-1 critical points of $\varphi$. The same holds for a nonorientable manifold, with $\R^2$ replaced by the M\"obius band and the bound on $S$ reduced by one.
\end{cor}

\begin{proof} We assume $n\ge2$, for otherwise the corollary is obvious (with the domain of the flow interpreted as $\R^1$ when $n=0$). Then Example~\ref{handles}(a) exhibits $V$ as the result of adding 2-handles at infinity to a manifold $U$ obtained from $\R^n$ by adding 1-handles at infinity along radial rays. It suffices to show that $U\times \R$ is diffeomorphic to the hypothesized manifold supporting the flow. This is easy when $U$ is orientable, since $n+1\ge 3$ so we can assume the 1-handles are in a standard position. If there are any nonorientable 1-handles, split $\R^{n+1}$ in half along a standard hyperplane, with the 1-handles positioned so that each nonorientable 1-handle has one foot on each side, and the orientable 1-handles all lie on one side. The complement of the hyperplane is then orientable. We can cut along the hyperplane and then repair the cut by attaching a 1-handle, which we view as the unique nonorientable member of a collection of 1-handles attached at infinity to $\R^{n+1}$. Now $U\times \R$ has the required nonorientable form.
\end{proof}

\begin{examples}\label{low} This corollary applies to every manifold of dimension two or less, and to every open 3-manifold. It applies to many 4-manifolds, for example every Stein surface, and to any 4-manifold after a suitable 1-complex is removed (the cocores of the 3- and 4-handles). The proof generalizes to allow arbitrary 2-handles at infinity, covering all of the exotic smooth structures on handlebody interiors discussed in Example~\ref{handles}(c), with the size of $S$ bounded by the number of 1-handles (less one in the nonorientable case). Again, we obtain infinitely many topologically but not smoothly equivalent flows when the quotient has nontrivial 2-homology. Every closed 5-manifold $X$ contains a separating 4-manifold for which the corollary applies to both components of the complement, with $S$ finite. If $X$ is simply connected, then both components are quotients of flows on $\R^6$. (Put a self-indexed Morse function $\varphi$ on $X$. If $X$ is simply connected, its index one critical points can be traded for some of index three, e.g.\ \cite{Mi}, and the same with Morse function $5-\varphi$ trades index four for two. Cut at $\varphi^{-1}(\frac52)$, and use $5-\varphi$ on the top part.)
\end{examples}


\section{Higher-index handles}\label{Higher}

We now generalize the main construction of the previous section to create invisible handles with indices larger than two. We expand Theorem~\ref{main} to include higher-index handles, provided that all  invisible handles are attached in a single layer. The rest of the previous section generalizes easily, although the previous discussion of Casson handles leads to a deeper investigation of topologically but not smoothly equivalent flows.

For higher-index handles, we will allow fixed points in our flow, raising the question of when the quotient can have a manifold structure near a fixed point. To retain the Hausdorff condition, we need all orbits near the fixed set $F$ to be circles or points. If we make the further assumption that these orbits are generated by a smooth circle action near $F$, then $F$ is a smooth submanifold. Choosing an $S^1$-invariant metric near $F$, we use the exponential map to identify a tubular neighborhood of $F$ with that of the 0-section of its normal bundle $\nu$, with the circle acting linearly on the fibers. For the quotient to be a manifold, each meridian sphere to $F$ should descend to a meridian sphere of its image in the quotient. This is a strong requirement, but is satisfied when the meridian is $S^3$ with the Hopf circle action. (One could also consider other Seifert circle actions on $S^3$.) Thus, we assume $F$ has codimension four with the normal circle action modeled by scalar multiplication on $\C^2$, which identifies $\nu$ as a Hermitian 2-plane bundle. (In our examples, the bundle will be trivial by construction.) The quotient map on each normal fiber is then modeled by a smooth (in fact, quadratic) map $h\co \C^2\to\C\times\R$, given by
$$h(z,w)=(2z\overline{w},|z|^2-|w|^2).$$
(It is easy to verify that $h$ maps the 3-sphere of radius $r$ to the 2-sphere of radius $r^2$, and each complex line through the origin to a real ray from the origin, collapsing each $S^1$-orbit to a point. Since each point preimage has constant values of $|z|^2\pm|w|^2$, it is easily seen to be a single $S^1$-orbit as required.) The standard $\rm{U}(2)$-action on $\C^2$ descends through $h$ to $\SO(3)$. Thus, the normal bundle $\nu$ descends to an $\SO(3)$-vector bundle $h_*\nu$ whose 0-section has a neighborhood identified with a neighborhood of the image $F_*$ of $F$ in the quotient space. This smooths the quotient manifold, with the submanifold $F_*$ having codimension three. The smoothing depends on our choice of invariant metric. (Changing how orthogonality pairs complex lines in a fiber of $\nu$ changes which pairs of rays in the quotient fit together as smooth lines.) However, any two invariant metrics are connected by a path of invariant metrics, so the smoothing on the quotient is unique up to a topological isotopy rel $F_*$ that is smooth off of $F_*$. Our quotient map can be thought of as a complex analog of a 2-fold covering map branched along a codimension-2 submanifold, where the latter quadratic map, given by $z\mapsto z^2$ in the normal directions, is viewed as taking the quotient of an $S^0$-action that collapses each meridian $S^1$ to $\R P^1$. Another quadratic comparison map is $z\mapsto |z|^2$ normal to a codimension-2 fixed set, which collapses out the obvious circle action to create a smooth manifold with boundary.

We can now generalize our invisible 2-handles. First note that our prototype for these handles comes from the Hopf map $S^3\to S^2$, by removing a closed ball from the domain to change it to $\R^3$. Specifically, if we take $U\times\R$ to be a neighborhood of a section over the northern hemisphere of $S^2$, then the circle bundle over the southern hemisphere is the invisible 2-handle. We will similarly construct a prototype of an invisible $k$-handle, $k>2$, using the $(k-2)$-fold suspension of the Hopf map, i.e.,  its join with the identity on $S^{k-3}$. This procedure yields a circle action on $S^{k+1}$, with fixed-point set $S^{k-3}$ having codimension four and the description given in the previous paragraph. Replacing the Hopf action on $S^3$ by this higher dimensional $S^1$-space extends our construction to higher-index handles.

\begin{thm}\label{high} Suppose that an $n$-manifold $V$ is obtained from another $n$-manifold $U$ by attaching a layer of handles at infinity, with indices at least two. Then there is a smooth flow on $U\times\R$ with quotient $V$ and quotient map homotopy equivalent to the embedding $i\co U\hookrightarrow V$. The fixed-point set consists of a properly embedded $S^{k-3}\times\R^{n-k}$ for each $k$-handle (where $S^{-1}=\emptyset$), with a neighborhood on which the quotient map is modeled by $h\times\id_{S^{k-3}\times\R^{n-k}}$, for $h\co\C^2\to\R^3$ as above.
\end{thm}

\begin{proof} First, we build a model invisible compact $k$-handle $H_k$ with $k=n\ge3$. Starting from the map $h\times \id_{\R^{k-2}}\co \C^2\times \R^{k-2}\to\R^3\times\R^{k-2}$, restrict and radially project to get $h_k\co S^{k+1}\to S^k$, the quotient map of the circle action on the unit sphere in $\C^2\times \R^{k-2}$ (where the circle acts trivially on $\R^{k-2}$). Topologically, this is the join of $h|S^3$ with the identity on $S^{k-3}\subset\R^{k-2}$, so the fixed set is $S^{k-3}$ with a local model of the form $h\times\id_{S^{k-3}}$. We alternatively view $S^{k+1}$ as the join of $S^1$ with $S^{k-1}$: Let $D\subset\C$ be the disk of radius $\frac12$ about 0, and let $H_k$ denote $S^{k+1}\cap(D\times\C\times\R^{k-2})$, an equivariant tubular neighborhood in $S^{k+1}$ of the unit sphere $S^{k-1}\subset\{0\}\times\C\times\R^{k-2}$ containing the fixed $S^{k-3}$. This will be our lifted $k$-handle. Its complement $S^{k+1}-\int H_k$ is a tubular neighborhood in $S^{k+1}$ of the unit circle of $\C\times\{0\}\times\{0\}$, with free circle action identifying it as a (necessarily trivial) circle bundle over a disk  $\Delta^k\subset S^k$. Thus, $h_k(H_k)=S^k-\int\Delta^k$ can be viewed as a $k$-handle attached to $\Delta^k$. But the local slice $\R_+\times\C\times\R^{k-2}$ of the circle action intersects $S^{k+1}-\int H_k$ in a local section $\sigma_k$ projecting diffeomorphically to $\Delta^k$, so its intersection $\Sigma_k=\partial\sigma_k$ with $\partial H_k$ projects diffeomorphically to $\partial\Delta^k$. Since $\Sigma_k$ has the form $S^{k+1}\cap(\{\frac12\}\times\C\times\R^{k-2})$, it follows that the pair $(H_k,\Sigma_k)$ is diffeomorphic to $(S^{k-1}\times D^2,S^{k-1}\times \{1\})$, so $H_k$ is an invisible $k$-handle with attaching sphere $\Sigma_k$, attached to a tubular neighborhood of the section $\sigma_k$ in $S^{k+1}-\int H_k$.

It is now routine to adapt the proof of Theorem~\ref{main}. The annulus $H_k\cap(\R_+\times\C\times\R^{k-2})$ intersects $\partial H_k$ transversely in $\Sigma_k$ and projects diffeomorphically to its image in the disk $h_k(H_k)$, so we can use it to glue the invisible open $k$-handle $\int H_k\times\R^{n-k}$ to $U\times\R$ as necessary (with the $k=2$ case reducing to the invisible 2-handles of the previous proof).
\end{proof}

As in the previous section, we immediately obtain a plethora of new examples.

\begin{cor}\label{highapps} (a) If an $n$-manifold $V$ is obtained from a contractible manifold $U$ by attaching a layer of handles at infinity with indices at least two, then $V$ is the quotient of some flow on $\R^{n+1}$.

(b) There is a flow on $\R^{n+1}$ with quotient a simply connected, open $n$-manifold $V$ realizing any preassigned  countably generated, free abelian homology groups $H_k(V)$, $2\le k<n$, any bilinear intersection pairings between  $H_k(V)$ and $H_{n-k}(V)$ ($2k\ne n$), and any skew-symmetric (for $n\equiv 2$ mod 4) or even symmetric (for $n$ divisible by 4) form in the middle dimension. If the homology is finitely generated, then $V$ is the interior of a compact manifold with boundary.

(c) For fixed $k,n$ with $1\le k\le n-3$ and any homomorphism $f\co G\to H$ of countably generated abelian groups, with free abelian cokernel $F$ that vanishes if $k=1$, 
there is an oriented $(n+1)$-manifold with a flow for which the quotient is an $n$-manifold and the quotient map induces $f$ on $k$-homology.

(d) For fixed $k,n$ with $1\le k\le n-3$, every countably generated abelian group is $H_k(V)$ for some quotient $n$-manifold $V$ of a flow on $(\R^{k+1}-S)\times\R^{n-k}$ ($S$ discrete and closed). For $k>1$, $V$ is simply connected, and for $k=1$ every countably presented group is realized as $\pi_1(V)$.

(e) For $n\ge 5$, suppose a simply connected $n$-manifold $V$ has a proper Morse function $\varphi\co V\to[0,\infty)$ with critical points of index at most three, and only finitely many of index one. If $V$ admits a spin structure, then it is the quotient of a flow on $W=(\R^3-S)\times\R^{n-2}$, where $S$ is a discrete, closed subset no larger than the set of index-2 critical points of $\varphi$. If $V$ does not admit a spin structure, the same holds with the size of $S$ reduced by one but the twisted $\R^{n-1}$-bundle over $S^2$ end-summed onto $W$.
\end{cor}

\begin{proof} These statements follow as do their analogs in the previous section, with (a) being immediate. For (b), it suffices to inductively construct framed links in $\partial D^n$ as in Example~\ref{handles}(b), by keeping $(n-1)$-handles isolated and otherwise tubing together meridians (which have codimension at least two, so the tubes do not collide with anything and respect orientations). For $n=2k$ divisible by 4, we can arrange any even self-intersection numbers by fiber summing the $k$-handles with copies of the tangent bundle of $S^k$, which has Euler number two. For (c), build $U$ with $H_k(U)\cong G$ from $\R^n$ by adding a $k$-handle at infinity for each generator and then a $(k+1)$-handle at infinity for each relator. (Again this works since $k+1\le n-2$ so we can tube together copies of generating spheres.) Add a layer of handles to $U$ with indices $k+1$ and (if $k>1$) $k$ to get $V$ with $H_k(V)\cong f(G)\oplus F\cong H$, then apply Theorem~\ref{high}. Part (d) is similar. For (e), consider a simply connected sublevel set of $\varphi$ containing all of the index-1 critical points. By standard Morse theory \cite{Mi}, we can cancel some of these critical points with minima, and trade the rest for index-3 critical points. Thus, without loss of generality, we can assume that all critical points but the unique minimum have index two or three. As with Corollary~\ref{all}, we obtain a domain $U\times\R$ consisting of $\R^{n+1}$ with (possibly infinitely many) 2-handles radially  attached, standardly in the spin case, and reduce to that case by drilling out a suitable $(n-1)$-plane.
\end{proof}

\begin{Remarks}\label{highapprem}(i) Some odd forms in (b) ($n=2k=4m$) can also be realized, by end-summing $V$ with copies of a neighborhood of $\C P^m\subset\C P^k$ that lifts to $\R^{n+1}$ by Calcut's method, to add $\langle \pm 1\rangle$ summands.

 (ii) The two conclusions in (e) are not mutually exclusive. If $V$ fails to be spin but the twisted bundle glued to $W$ comes from a 2-handle in $V$ disjoint from the 3-handles, then the 2-handle can be lifted invisibly to realize $V$ as the quotient of a flow on $W$. (On the other hand, we cannot realize a spin quotient manifold $V$ of a flow on a nonspin manifold $W$, since the vanishing Stiefel-Whitney class $w_2(V)$ would pull back to zero on the complement of the fixed set of codimension greater than two, and then extend trivially over $W$ to show $w_2(W)=0$.) The finiteness condition in (e) is stronger than necessary. It is enough for $V$ to be the interior of a handlebody $H$ with handles of index at most three, such that $H$ (including its boundary) has endpoint compactification that is locally 1-connected at each endpoint. This allows the required handle trading with proper infinite collections of handles.
\end{Remarks}

Combining Theorem~\ref{high} with classical smoothing theory yields various new phenomena related to topological versus smooth equivalence of flows.

\begin{examples}\label{spheres} (a) Let $\Sigma$ be an exotic $n$-sphere, $n\ne 4$. (These exist for most values of $n\ge7$. There are finitely many in each dimension, comprising groups under connected sum that are well understood \cite{KM}.) One can realize $\Sigma$ as a 0-handle with an $n$-handle attached by an exotic diffeomorphism $\varphi$ of the boundary $(n-1)$-sphere. By Theorem~\ref{high}, $\Sigma$ is the quotient of some flow on $\R^{n+1}$. Since $\varphi$ extends homeomorphically over the 0-handle by coning, we obtain a homeomorphism from $S^n$ to $\Sigma$ that lifts to a topological equivalence between the corresponding flows on $\R^{n+1}$. The resulting finite family of topologically equivalent but smoothly inequivalent flows on $\R^{n+1}$ is qualitatively different from those encountered previously since it is stable under Cartesian product with $\R$: The h-Cobordism Theorem implies that whenever $\Sigma\times\R$ and $\Sigma'\times\R$ are diffeomorphic, so are $\Sigma$ and $\Sigma'$. On the other hand, exotic spheres are stably parallelizable, so Mazur's work \cite{Ma} implies that $\Sigma\times \R^{n+1}$ is always diffeomorphic to $S^n\times\R^{n+1}$. Thus, it is unclear whether smooth inequivalence is preserved under iterated stabilization. However, a weaker notion of inequivalence is preserved. High-dimensional smoothing theory classifies smoothings up to isotopy. Two smoothings on a fixed topological manifold are {\em isotopic} if there is a diffeomorphism between them that is topologically isotopic (homotopic through homeomorphisms) to the identity. Isotopy classes of smoothings in dimensions five and higher are classified by obstructions in cohomology, and Cartesian product with $\R$ induces a bijection of isotopy classes. In particular, nondiffeomorphic (equivalently, nonisotopic) smoothings on $S^n$ remain nonisotopic after iterated stabilization, although they eventually become diffeomorphic. In conclusion, if $\Sigma$ and $\Sigma'$ are nondiffeomorphic manifolds homeomorphic to $S^n$, then the above construction gives smooth flows on $\R^{n+1}$, that stabilize to flows on $\R^{n+1+m}$ with quotients diffeomorphic to $\Sigma\times\R^m$ and $\Sigma'\times\R^m$, respectively. The above also gives a topological equivalence $f\co\R^{n+1}\to\R^{n+1}$ stabilizing to one on $\R^{n+1+m}$ that allows us to topologically identify the quotients with each other and with $S^n\times\R^m$. For $m\le1$, there can be no smooth equivalence between the flows (since the quotients are nondiffeomorphic), and for general $m$ there is at least no smooth equivalence that is homotopic through topological equivalences to $f\times\id_{\R^m}$ (since the smoothings on the quotient $S^n\times\R^m$ are nonisotopic).

(b) In contrast to (a), all of the topologically equivalent but smoothly inequivalent flows that we have seen so far on 5-manifolds have been unstable under product with $\R$, with a smooth equivalence after product with $\R$ that is isotopic as above to the given topological equivalence. This reflects indistinguishability of the relevant smoothings of 4-manifolds by classical smoothing obstructions: The unique obstruction below dimension seven is $H^3(V;\Z/2)$, which vanished in our examples from Section~\ref{2h}. However, we can also construct examples with quotients distinguished by this obstruction, so stable up to isotopy as in (a). There are uncountably many diffeomorphism types of smoothings on $S^3\times\R$, in each of the two stable isotopy classes (the latter detected by $H^3(S^3\times\R;\Z/2)\cong\Z/2$). We can assume each smoothing is standard near some line $\{p\}\times\R$, whose complement $R$ is a possibly exotic smoothing of $\R^4$. (See Chapter~8 of Freedman and Quinn \cite{FQ} for background.) The original manifold is then exhibited as $R$ with a 3-handle attached at infinity, so by Corollary~\ref{highapps}(a) it is the quotient of some flow on $\R^5$. We obtain uncountably many such smoothly inequivalent flows, which are all topologically equivalent since they come from $\R^4$ with a 3-handle attached along a topologically standard annulus. After Cartesian product with $\R^m$, $m\ge1$, we are left with two isotopy classes of smoothings on $S^3\times\R^{1+m}$ (which are diffeomorphic for large $m$ by Mazur's argument). The corresponding flows on $\R^{5+m}$ fall into two classes, under smooth equivalence homotopic as above to the given topological equivalence. (They are equivalent when not distinguished by the smoothing obstruction, since then the obvious diffeomorphism of the 3-handles extends smoothly over $R\times\R^m$, isotopic rel the 3-handle to the given homeomorphism, and this whole description lifts to $\R^{5+m}$ since the flow is trivial over $R\times\R^m$.)

(c) It is unknown whether exotic smoothings exist on $S^4$ (the smooth 4-dimensional Poincar\'e Conjecture). If any exist, they require more than two handles in any handle decomposition, so (a) does not apply directly. However, we could still exhibit them as a 4-handle attached to an exotic 4-ball, or a 4-handle attached at infinity to a contractible open 4-manifold, so they would still arise as quotients of flows on $\R^5$. We would again have topologically equivalent flows on $\R^5$ that are unstably smoothly inequivalent.

(d) The method of Examples~\ref{examples}(b) and \ref{low} generalizes to higher-index exotic handles. In those examples, we saw that invisible 2-handles on quotient 4-manifolds can be replaced by invisible exotic 2-handles, namely Casson handles thought of as 2-handles attached exotically at infinity. A single invisible 2-handle can sometimes be replaced by Casson handles yielding uncountably many diffeomorphism types (via \cite{MinGen}~Theorem~7.1 or 7.7), so we get uncountably many topologically equivalent but smoothly inequivalent flows. Similarly, (b) and (c) above result in invisible exotic 3-handles and possibly 4-handles on quotient 4-manifolds, with the 3-handles sometimes yielding uncountably many diffeomorphism types. The only difference from the case of Casson handles is that these will be handles attached at infinity after the smooth structure of the affected 4-manifold is possibly changed by end-sum with an exotic $\R^4$. (Compare with Example~\ref{handles}(c), which uses the fact that a Casson handle interior is diffeomorphic to $\R^4$. The analogous statement for 3- and 4-handles may fail.) However, this does not cause difficulties since the exotic $\R^4$ will still be invisible when lifted to the manifold supporting the flow. These 4-dimensional phenomena will vanish under product with $\R$, except for the $\Z/2$ arising in (b). However, in higher dimensions, (b) and  (a) above yield an exotic invisible handle of index three, and finitely many of most indices seven and higher. The smoothings realized on a given $k$-handle correspond bijectively to the obstruction group $H^k(D^k\times\R^{n-k},\partial D^k\times\R^{n-k};\Theta_k)\cong\Theta_k$, where $\Theta_k$ is the group of smoothings of $S^k$ (or $S^k\times\R^2$ if $k=3$) up to isotopy. Thus, we can use the classical invariants to classify the resulting smoothings up to isotopy, obtaining families (finite if there are finitely many handles) of flows related by topological equivalences not isotopic to smooth ones, as in (a). For example, we can find such families of flows on $\R^7$ or $\R^{15}$ realizing any skew-symmetric form on a nontrivial, countably generated, free $\Z$-module as the intersection form of the quotient manifold. In general, more work is required to determine when the quotient manifolds are nondiffeomorphic, to guarantee that the flows are smoothly inequivalent.

(e) The other side of classical smoothing theory is a complete (when $n>4$) set of cohomological obstructions to the existence of smooth structures on topological manifolds. In our setting, these give rise to {\em topological} flows (1-parameter groups of homeomorphisms) on smooth manifolds, that are not conjugate by any homeomorphism to a smooth flow on a manifold (although locally they are, near any orbit). Suppose, for example, that $n=2k>4$ with $k$ even. Let $U$ be obtained by plumbing together eight copies of the tangent bundle of $S^k$ according to the $E_8$ graph. This is a well-known parallelizable $n$-manifold with the homotopy type of a wedge of eight $k$-spheres and intersection form $E_8$, and it is the interior of a compact manifold whose boundary is the Milnor exotic $(n-1)$-sphere $\Sigma$. We think of $U$ as the interior of a handlebody with one 0-handle and eight $k$-handles. Then $U\times\R$ has this same handle description, only the bundles over $S^k$ are trivial and the handles cannot link in these dimensions, so $U\times\R$ is diffeomorphic to $(\R^{k+1}-8{\rm\ points})\times\R^k$. Let $X=U\cup\{\infty\}$ be the one-point compactification of $U$. (Equivalently, we cone off the boundary $\Sigma$.) Then $X$ is a closed topological $n$-manifold that admits no smoothing (since there is no smooth, almost-parallelizable $n$-manifold with signature eight). Its stabilization $X\times\R$ is also unsmoothable, since its existence obstruction in $H^n(X;\Theta_{n-1})\cong\Theta_{n-1}$ (which is given by $\Sigma$) is stable under product with $\R$. However, if we remove the line $\{\infty\}\times\R$, we recover $U\times\R$, which has a smoothing to which the obvious topological flow on $X\times\R$ restricts smoothly. Applying Theorem~\ref{high} to an $n$-handle of $X$ that lies smoothly in $U$, we obtain a topological flow on a manifold homeomorphic to $U\times\R\approx(\R^{k+1}-8{\rm\ points})\times\R^k$ that cannot be topologically conjugate to a smooth flow on a smooth manifold since its quotient is $X$. (Smoothing the flow would give a smooth circle action near the fixed set. We could model this as usual by a $\rm{U}(2)$-bundle to push the smoothing down to $X$.) This flow stabilizes to unsmoothable flows on $(\R^{k+1}-8{\rm\ points})\times\R^m$ for all $m\ge k$.
\end{examples}

Corollary~\ref{all} showed that if an orientable 4-manifold is the interior of a handlebody lacking 3- and 4-handles, then it is the quotient of a flow on $\R^5$ minus parallel 3-planes. However, manifolds with 3-homology were excluded. We now show that a much larger range of  4-manifolds satisfies the same conclusion, and that the result is unaffected by putting exotic smooth structures on these manifolds.

\begin{cor}\label{last} Suppose $U$ is a 4-dimensional handlebody interior with handles only of index zero and one. Let $V$ be obtained from $U$ by topologically adding a layer of handles at infinity and then putting any smooth structure on the result. If $V$ is orientable, then it is the quotient of a smooth flow on $(\R^2-S)\times\R^3$ for some discrete, closed $S$. The same holds for nonorientable $V$ with $\R^2$ replaced by a M\"obius band. If $U$ is instead a contractible 4-manifold, and $V$ is simply connected, then $V$ is the quotient of a flow on $\R^5$.
\end{cor}

\noindent For example, $V$ could be obtained from a 4-ball or connected sum of copies of $S^3\times S^1$, by multiply puncturing, adding 1-handles and then 2-handles at infinity (which can be assumed to avoid the 3-handles of the obvious handle decomposition). One expects many smoothings on $V$, and hence topologically but not smoothly equivalent flows.

\begin{proof} It suffices to topologically parametrize each 2-handle at infinity so that its cocore $\{0\}\times\R^2\subset\R^2\times\R^2$ is smoothly embedded in $V$. Then the previous method applies: Removing the smooth cocore of each handle of index two or more will give $U$ with extra 1-handles and a possibly exotic smooth structure, but after product with $\R$ the smoothing becomes standard. To smooth a 2-handle cocore $C$, decompose it as an infinite cell complex. We can then assume $C$ is smooth except at the center point of each 2-cell (\cite{FQ}~Chapter~8). Connect these singular points by a properly embedded ray $\gamma$ in $C$. Then $V-\gamma$ is diffeomorphic to $V$. (First remove the singular points from $V$, then smoothly isotope the remaining properly embedded open intervals of $\gamma$ to agree with those of a smooth ray $\gamma'$ in $V$, showing that $V-\gamma\approx V-\gamma'\approx V$.) But in $V-\gamma$, the surface $C-\gamma$ is exhibited as a smooth cocore of the topological 2-handle as required.
\end{proof}


\section{Further questions}\label{Further}

We have now seen that while 1-handles are tightly constrained by $\pi_1$-surjectivity of quotient maps, a layer of higher-index handles can be added to a suitable quotient manifold without disturbing the manifold supporting the flow. Surjectivity is the only constraint on $\pi_1$, and the higher homology and intersection pairing can be chosen flexibly, without obvious constraints. This leads to the following questions:

\begin{quess}\label{q} (a) Is every simply connected $n$-manifold realized as the quotient of a flow on $\R^{n+1}$?

(b) Is every oriented $n$-manifold realized as the quotient of a flow on $(\R^2-S)\times\R^{n-1}$ for some discrete, closed subset $S$?
\end{quess}

\noindent There is also a nonorientable analog of (b) with $\R^2$ replaced by the M\"obius band. As we have seen (Examples~\ref{low}), both versions of (b) have an affirmative answer for $n\le2$, for open 3-manifolds, and for various higher dimensional manifolds. The answer to (a) is also affirmative for homotopy spheres (Examples~\ref{spheres}), so for closed manifolds of dimensions three and below, and for contractible manifolds with a single layer of handles attached at infinity (Corollary~\ref{highapps}(a)). For each question, the set of affirmative open examples is closed under end sum and hence (for $n\ne2$) connected sum (the latter by adding an invisible $(n-1)$-handle onto the 1-handle at infinity). Compared with (b), Question (a) has the additional complication that 1-handles cannot be lifted invisibly. In dimensions five and up, 1-handles can usually be traded for 3-handles (cf.\ Remark~\ref{highapprem}(ii)), but dimensions three and four are harder. One might try to localize the index-1 critical points within a contractible subhandlebody to apply Corollary~\ref{highapps}(a), although this leads into difficult issues regarding group presentations already in the case with only finitely many critical points. Beyond this issue, to make further progress on any of these questions, one needs to be able to add multiple layers of invisible handles. This poses difficulties, since the attaching annuli of the second layer of handles need not lift. For example, an attaching sphere of the second layer may run over an invisible handle of the first layer. The second layer may also need to avoid the fixed points of the first layer, since otherwise it is more difficult to arrange the flows to fit together. Fortunately, there is a prototype for iterated handle attaching: The handles of $\C P^m$ with index at least two lift invisibly in Calcut's example, and those with index exceeding two attach along spheres in homotopy classes that do not lift. Perhaps some insight can be gained by studying these handles. Note that Calcut's flow has no fixed points, so the example may be useful in constructing fixed-point free invisible handles of even index four and higher. An intermediate question that tests iterated handle attaching is the following:

\begin{ques} Can a quotient of a flow on $\R^{n+1}$ have torsion in homology?
\end{ques}

\noindent Recall that by Corollary~\ref{highapps}(d), every countably generated abelian group is the $k$-homology of some manifold quotient of a flow on $(\R^{k+1}-S)\times\R^{n-k}$ ($S$ discrete and closed), so torsion can be created from some manifolds, but torsion from a contractible manifold requires two layers of invisible handles.

 The techniques of this paper should also be useful for studying higher dimensional foliations. The quotient map in this broader context will still be a $\pi_1$-surjection by \cite{CGM}, but one can obtain invisible 4-handles for 3-dimensional foliations, for example, from the quaternionic Hopf fibration $S^7\to S^4$. This realizes all symmetric intersection forms by quotient 8-manifolds of 3-dimensional foliations on $\R^{11}$. (Arbitrary self-intersections can be realized since there is a bundle over $S^4$ with Euler number one.)
 
 \section*{Acknowledgements}
The author would like to thank Jack Calcut for helpful discussions. Parts of this work were supported by NSF grant DMS-1005304.

\end{document}